\documentclass[12pt]{amsart}

\usepackage{color}
\usepackage{amsthm}
\usepackage{amsxtra}
\usepackage{amssymb}
\usepackage{graphicx}  
\newcommand{\be}{\begin{equation}}
\newcommand{\ee}{\end{equation}}

\newcommand{\CI}{{\mathcal C}^\infty }
\newcommand{\CIc}{{\mathcal C}^\infty_{\rm{c}} }

\newcommand{\CC}{{\mathbb C}}

\newcommand{\Q}{{\mathbb Q}}
\newcommand{\RR}{{\mathbb R}}

\newcommand{\TT}{{\mathbb T}}

\newcommand{\ZZ}{{\mathbb Z}}

\newcommand{\supp}{\operatorname{supp}}

\newcommand{\rest}{\!\!\restriction}

\theoremstyle{plain}

\newtheorem{thm}{Theorem}
\newtheorem{prop}{Proposition}[section]
\newtheorem{cor}[prop]{Corollary}
\newtheorem{lem}[prop]{Lemma}

\theoremstyle{definition}

\newtheorem{rem}[prop]{Remark}

\numberwithin{equation}{section}

\def\bbbone{{\mathchoice {1\mskip-4mu {\rm{l}}} {1\mskip-4mu {\rm{l}}}
{ 1\mskip-4.5mu {\rm{l}}} { 1\mskip-5mu {\rm{l}}}}}

\def\squarebox#1{\hbox to #1{\hfill\vbox to #1{\vfill}}}

\newcommand{\Id}{{I}}

\usepackage{amsxtra}

\ifx\pdfoutput\undefined
  \DeclareGraphicsExtensions{.pstex, .eps}
\else
  \ifx\pdfoutput\relax
    \DeclareGraphicsExtensions{.pstex, .eps}
  \else
    \ifnum\pdfoutput>0
      \DeclareGraphicsExtensions{.pdf}
    \else
      \DeclareGraphicsExtensions{.pstex, .eps}
    \fi
  \fi
\fi

\title[Control for Schr\"odinger operators on tori]
{Control for Schr\"odinger operators on tori}
\author[N. Burq]{Nicolas Burq}
\address{Universit{\'e} Paris Sud,
Math{\'e}matiques,
B{\^a}t 425, 91405 
Orsay Cedex, France, 
and Ecole Normale Sup\'erieure, 45, rue d'Ulm, 75005 
Paris,  Cedex 05, 
France}
\email{Nicolas.burq@math.u-psud.fr}
\author[M. Zworski]{Maciej Zworski}
\address{Mathematics Department, University of California, Berkeley, CA 94720, USA}
\email{zworski@math.berkeley.edu}

\usepackage{amssymb}
\usepackage{amsmath, amsthm, amsopn, amsfonts}

\def\11{{\rm 1~\hspace{-1.4ex}l} }
\def\R{\mathbb R}

\def\T{\mathbb T}

\begin{document}

\begin{abstract}
A well known result of Jaffard states that an 
arbitrary region on a torus controls, in the $ L^2 $
sense, solutions of the free stationary 
and dynamical Schr\"odinger equations.
In this note we show that the same result is valid
in the presence of a potential, that is for 
Schr\"odinger operators, $ - \Delta + V $, $ V \in \CI $.
%In the dynamical case $ V $ can be made dependendent on
%time.
\end{abstract}   

\maketitle   

\section{Introduction}   
\label{in}

%Consider the two dimensional torus $\T^2= \RR^2 / ((A \ZZ) \times ( B
%\ZZ ))$, $ A, B  \in \RR \setminus \{ 0 \}$. 
We show how simple methods introduced in
\cite{BZ3}, \cite{BZ31}, \cite{Ma} (see also \cite{HHM} and \cite{Mac}) for the study of 
the equation 
\[   ( - \Delta - \lambda ) u ( z) = f ( z ) \,, \ \    z \in 
\T^2 : = \RR^2 / A \ZZ \times B\ZZ \,, \ \ A, B  \in \RR \setminus \{ 0 \}\,, \]
and the control of 
\[  i \partial_t u ( t, z ) =  - \Delta   u ( t , z) \,, \
\ z \in \T^2 \,.
\]
can be adapted to obtain similar results for the equations
\begin{equation}
\label{eq:st}   ( - \Delta + V ( z ) - \lambda ) u ( z) = f ( z ) \,, \ \    z
\in \T^2\,,\end{equation}
and 
\begin{equation}
\label{eq:dy}
 i \partial_t u ( t, z ) = ( - \Delta + V( z ) ) u ( t , z) \,, \
\ z \in \T^2 \,, \end{equation}
where $ V \in \CI ( \T^2 ) $ is a smooth real  valued
potential.

The first theorem concerns solutions of the stationary 
Schr\"odinger equation and is applicable to high
energy eigenfunctions:

\begin{thm}
\label{t:1}
Let $ \Omega \subset \T^2 $ be any open set.
There exists a constant $ K = K ( \Omega ) $, depending only 
on $ \Omega $, such that 
for any solution of \eqref{eq:st} we have
\begin{equation}
\label{eq:t1}
\| u \|_{ L^2 ( \T^2) } \leq K \left( \| f \|_{ L^2 ( \T^2) }  + \| u \|_{ L^2 ( \Omega ) } \right) \,. 
\end{equation}
\end{thm}

This means that $ u $ on $ \T^2 $ is {\em controlled} by $ u
$ in $ \Omega$, in the $ L^2$ sense. 
The next result, which is in fact more general, concerns 
the dynamical Schr\"odinger equation:
% in which we allow the potential to be time dependent.

\begin{thm}
\label{t:2}
Let $ \Omega \subset \T^2$ be any (non empty) open set and let $ T > 0 $.
There exists a constant $ K = K ( \Omega, T  ) $, depending only 
on $ \Omega $ and $ T $, such that 
for any solution of \eqref{eq:dy} we have
\begin{equation}
\label{eq:t2} \| u ( 0 , \bullet ) \|^2_{ L^2 ( \T^2) } \leq K  \int_0^T \|
u ( t , \bullet )  \|^2_{ L^2 ( \Omega ) } dt \,. 
\end{equation}
\end{thm}
An estimate of this type is called {\em an observability}  result.
Once we have it, the  HUM
method  (see~\cite{Li}) automatically provides the following {\em
  control} result:
\begin{thm}
\label{t:3}
Let $ \Omega \subset \T^2$ be any (non empty) open set and let $ T > 0 $. For any $u_0 \in L^2(\T^2)$, there exists $f\in L^2((0,T)\times \Omega)$ such that the solution of the equation
$$ (i\partial_t +\Delta - V(z) ) u (t,z)= f \bbbone_{(0,T) \times
  \Omega}(t,z) \,,  \qquad  u ( 0 , \bullet )  = u_0\,, $$
satisfies 
$$ u ( T , \bullet )  \equiv 0 \,. $$
\end{thm}

By applying Theorem~\ref{t:2} to the initial data
$u(0,\bullet) = u$, it is easy to see that Theorem~\ref{t:1} follows
from Theorem \ref{t:2} and the Duhamel formula. As a consequence, we will 
restrict our attention to Theorem~\ref{t:2}.

In the case of $ V \equiv 0 $ 
the estimates \eqref{eq:t1} and~\eqref{eq:t2} were proved
by Jaffard \cite{Ja} and 
Haraux \cite{Ha} using Kahane's work \cite{Ka}
on lacunary Fourier series. 
%A simple perturbation argument
%shows that the estimates remain true if $\|V\|_{L^\infty}$ is small
%enough. 

For a presentation of control theory for the Schr\"odinger equation
we refer to \cite{Le} -- see also \cite{Bu},\cite{TW},  and \cite[\S 3]{BZ2}.

We conclude this introduction with comments about a natural class
of potentials for which the theorems above should hold. When 
$ V \in L^\infty $ and $ \| V \|_{L^\infty } \ll 1 $ a perturbation 
argument shows that \eqref{eq:t1} and \eqref{eq:t2} follow from 
results with $ V = 0 $. 

The methods of this paper can be 
extended to the case of $ V \in C^0 ( \TT^2 ) $ by first showing that
the constant in the high frequency estimate \eqref{eq:semi} is
independent of $ V $ for $ V $ in a bounded subset of $ L^\infty $
and then using approximation and a perturbation argument. The 
restriction that $ V $ is real is not essential but makes the 
writing easier as we can use the calculus of self-adjoint operators.

\medskip
\noindent
{\bf Conjecture.} {\em Theorems \ref{t:1},\ref{t:2},\ref{t:3} hold for 
$ V \in L^\infty ( \TT^2 ; \CC ) $. Theorems \ref{t:2} and \ref{t:3} 
hold for time dependent potentials $ V ( t , z ) \in L^\infty ( \RR
\times \TT^2 ; \CC ) $. }

\medskip
\noindent
{\sc Acknowledgments.} 
We would like to thanks Semyon Dyatlov, Luc Hillairet, and Claude Zuily for 
helpful conversations.
The first author acknowledges partial support from Agence Nationale de
la Recherche project ANR-07-BLAN-0250 and 
the second author acknowledges partial 
support by the National Science Foundation under the grant 
DMS-0654436. He is also grateful to Universit\'e de Paris-Nord
for its generous hospitality in the Spring 2011 when this paper
was written. 

\section{Preliminaries}
\label{pr} 
In this section we will recall the basic control result \cite{Bu92},\cite{BZ3}
for rectangles, and the normal form theorem based on 
Moser averaging method \cite{W}.

\renewcommand\thefootnote{\dag}%

The following result~\cite{Bu92} is related to some earlier control results 
of Haraux~\cite{Ha} and  Jaffard~\cite{Ja}\footnote{We remark that as noted in \cite{Bu92} 
the result holds for 
any product manifold $M= M_{x}\times M_{y}$, and the proof is essentially the same.}:
\begin{prop}
\label{p:1} Let $\Delta$ be  the Dirichlet, Neumann, or periodic
Laplace operator on the rectangle $R= [0, a]_{x} \times [0,b]_{y}$. 
Then for any open non-empty  $\omega\subset R$ of the form $
\omega= \omega_{x} \times [0,a]_{y}$ , there 
exists $C$ such that for any solutions of
\begin{equation}
(\Delta -z) u =f \ \text{ on $R$}, \ u \rest_{\partial R}=0
\end{equation}
we have
\begin{equation}
\label{eq:6.12}\|u\|^2_{{L^2(R)}}\leq C \left(\|f \|^2_{L^2([0,b]_{y}); H^{-1} (
[0,a]_{x} ) } +
\|u \|^2_{{L^2(\omega)}} \right)
\end{equation}
\end{prop}
\begin{proof}
We will consider the Dirichlet case (the proof is the same in the other two 
cases) and 
decompose $u,f$ in terms of 
the basis of $L^2([0,b])$ formed by the Dirichlet eigenfunctions
$e_{k}(y)=  { \sqrt {{2}/b}}\sin(2k\pi y/b)$,
\begin{equation}
u(x,y)= \sum_{k}e_{k}(y) u_{k}(x), \qquad f(x,y)= \sum_{k}e_{k}(y) f_{k}(x)
\end{equation}
we get for $u_{k}, f_{k}$ the equation
\begin{equation}\label{estres.1}
\left(\Delta_{x}-\left(z+\left({2k\pi}/{b}\right)^2\right)\right)u_{k}= f_{k},\qquad u_{k}(0)=u_k(1)=0
\end{equation}
We now claim that 
\begin{equation}
\label{eq:cont}
\|u_{k}\|^2_{{L^2([0,1]_{x})}}\leq C \left(\|f_k \|^2_{H^{-1}([0,1]_{x})} +
\|u_k \rest _{\omega_{x}}\|^2_{{L^2(\omega)}}\right) 
\end{equation}
from which, by summing the squares in $k$, we get~\eqref{eq:6.12}.

To see \eqref{eq:cont} we can use the propagation result below in dimension
one, but in this case an elementary calculation is easily available -- see
\cite{BZ3}.
\end{proof}

The next proposition is the dynamical version of Proposition \ref{p:1}. 
However we change the assumptions on $ u $.

\begin{prop}
\label{p:1bis} 
Let $ R = [ 0, a ]_x \times [0, b]_y $, and let  $ \omega = \omega_x
\times [0, b ] $, where $ \omega_x $ is an open subset of $ [ 0 , b ] $.
%\[  \omega_x = [ 0, \epsilon ) \cup ( 1 - \epsilon , 1 ] \,, \ \
%\epsilon > 0\,.   \]
Suppose that for $ W \in \CI ( \RR ) $, $ W ( x + a ) = W ( x ) $, 
\begin{gather}
\label{eq:1bis}
\begin{gathered}
 i \partial_t u ( t, x , y) = (- \Delta + W ( x ) ) u ( t, x, y ) \,, 
\,, 
\end{gathered}
\end{gather}
and that, for some $ \gamma \in \RR $, $ u $ satisfies the following periodicity condition:
\begin{equation}
\label{eq:1bisb}  u ( t, x + k a , y + \ell b ) = u ( t , x , y + k \gamma ) \,, \ \
k, \ell \in \ZZ \,. 
\end{equation}

Then 
\begin{equation}
\label{eq:1bis1}
 \| u ( 0 , \bullet ) \|^2_{ L^2 (R ) } \leq K  \int_0^T \|
u  ( t, \bullet ) \|^2_{ L^2 ( \omega ) } dt   \,. 
\end{equation}
\end{prop} 

\medskip
\begin{rem}
Unitarity of the propagator $ \exp (  - i t ( -
\Delta + W ) ) $ shows that the $ ( 0 , T ) $ range integration on the right hand side
of \eqref{eq:1bis} can be replaced by $ ( T' , T ) $ for any $  0 \leq T'
< T $.  Same statement is true in the case of \eqref{eq:t2}.
\end{rem}

\medskip

\begin{proof}
As in the proof of Proposition \ref{p:1} we reduce the estimate 
to an estimate in one dimension.

To do that we see that \eqref{eq:1bisb} implies that $ u $ is periodic
in $ y $ and hence can be expanded into a Fourier series:
\[  u ( t , x, y ) =  \sum_{ n \in \ZZ } e^{ - i t n^2 } u_n ( t , x )
e^{ 2 \pi i y / b } \,, \ \ u_n ( t, x ) := \frac 1 b \int_0^b u(t, x
, y ) e^{ - 2 \pi i y / b } dy \,. \]
The condition \eqref{eq:1bisb} now means that
\begin{gather*} u_n ( t , x + a ) = e^{ 2 \pi i \gamma n / b} u_n ( t, x ) = e^{
  2\pi i \gamma_n } u_n ( t, x ) \,,  \\  \gamma_n = \gamma n / b -
[\gamma n / b ] \,, \ \ 0 \leq \gamma_n < 1 \,,  
\end{gather*}
that is, the periodicity in $ x $ is replaced by a Floquet 
periodicity condition. 

Proposition~\ref{p:1bis}  then follows from 
Lemma~\ref{lem.2.3} below.
\end{proof}

\begin{lem} \label{lem.2.3}
Let $ \omega_x \subset [ 0 , a ] $ be any open set.
Suppose that $ v $ solves 
$$ (i \partial_t - D_x^2 - W ( x ) ) v = 0  \,, \ \ W ( x + a ) = W (
x ) \,$$
and for some $ \alpha$, $ 0 \leq \alpha < 1 $, $ v$ satisfies
a Floquet periodicity condition, 
\[  v ( t, x + a ) = e^{ 2 \pi i \alpha } v ( x ) \,. \]
Then for any $ T $ there exists $ C$, independent of $ \alpha $, such
that 
\begin{equation} 
\label{eq:circ} \| v ( 0 , \bullet ) \|_{L^2 ( [ 0 , a ] ) }^2 \leq C \int_0^T \| v
( t, \bullet) \|^2_{L^2 ( \omega_x ) } dt \,. 
\end{equation}
\end{lem} 
\begin{proof}
We use the semi-classical approach developed by
Lebeau~\cite[Theorem 3.1]{Le} though 
the situation is
simpler here as we are dealing with internal controls in dimension
$1$. 

Writing 
$  w ( x ) := e^{ - 2 \pi i \alpha x / a } v  (x ) $, 
we obtain a periodic function $ w $ satisisfying
\begin{equation}
\label{eq.redu}  ( i \partial_t -  ( D_x + \beta )^2 - W ( x ) ) w = 0 \,, \ \
\beta := \frac{ 2 \pi \alpha} a \,. 
\end{equation}
The argument from \cite{Le} (used in \S\S \ref{se},\ref{fse}, 
below % for a similar argument in our setting 
-- see Remark~\ref{rem.leb}) applies and shows uniformity in 
$ \alpha $. % (notice that the method developed below also implies
            % this result (see Remark~\ref{rem.leb}).
For reader's convenience we provide more details in the appendix.
\end{proof}

Next we present a slight variation of the well known normal form result  -- see \cite{W}
where it was used in the case of Zoll manifolds (of which the circle is 
a trivial example). Our version can also be seen as a special case of the normal form in~\cite{De}

We start by introducing some notation:
we  have the spaces of standard pseudodifferential operators $\Psi^{m } ( \TT )$, $\Psi^{m } ( \TT^2 ) $ while 
\begin{equation}
\label{eq:cps}
  \CI \otimes \Psi^{m} := \CI ( \TT_x^1 ) \otimes \Psi^m ( \TT_y )
\,, \end{equation}
denotes semiclassical pseudodifferential  operators (of order m) in $ y $ depending smoothly 
on $ x $ as a parameter.

To makes things transparent we first present normal form results
for tori.
\begin{prop}
\label{p:2} 
Let $\chi \in \CIc ( \mathbb{R}^2) $ be equal to $ 0 $ in a
neighbourhood of $ \eta = 0 $. Suppose 
that $ V ( x, y ) \in \CI ( \TT^1 \times \TT^1 ) $. Then there
exist operators
%a semiclassical pseudodifferential operator on $ \TT^1 $, depending
%smoothly on $ x \in \TT^1 $, 
\[   Q( x, y , hD_y ) \in \CI \otimes \Psi^{0} \,, \ \ R ( x, y ,
hD_x, hD_y  )
\in \Psi^{0 } ( \TT^2 )  \,, \]
such that 
\begin{equation}
\label{eq.normale}
\begin{split}
& ({\Id} + h Q) \left(D^2_y + V ( x, y ) \right)\chi( hD_x,  hD_y)\\
& \ \ \ \ \ \ \ \ = (D_y^2 + V_0 ( x)) ({\Id} + hQ) \chi( hD_x, hD_y) + hR % ( x, y , hD_y) 
  \,, \end{split}
\end{equation}
where 
\[  V_0 ( x) = \frac 1 { 2 \pi} \int_{\TT^1} V ( x , y ) dy
\,.\]
\end{prop}
\begin{proof}
Indeed, we have 
\begin{multline}
\label{eq:2}
(\Id +h Q) \left(D^2_y + V ( x, y ) \right) \chi( h D_x, hD_y) -
(D_y^2 + V_0 ( x)) ( \text{ Id} + hQ) \chi( h D_x , hD_y) \\
= \Bigl( h[Q, D_y^2] + V(x,y)- V_0(x) + h R_1  \Bigr) \chi( hD_y)  \end{multline}
 with $R_1 \in \CI \otimes \Psi^{0 } $.
The pseudodifferental calculus shows that to obtain
\eqref{eq.normale}, 
it is enough to find $q \in \CIc( \T^2\times \mathbb{R}) $ such that 
\begin{equation}\label{eq.norm}
 \Bigl(-\frac 2 i \eta \,  \partial_y q (x, y, \eta) + V(x,y) - V_0(x)
 \Bigr) \chi(  \xi, \eta) = 0
 \end{equation}
 Since $\chi$ vanishes near $\eta=0$, we can find $\zeta \in \CIc (
 \mathbb{R} \setminus \{0\})$ equal to $1$ on
the support of $\chi$, and we can solve~\eqref{eq.norm} by taking
 \begin{equation}\label{eq.norm2}
 q(x, y, \eta) = \frac{ i \zeta (\eta)} { 2 \eta} \int_0^y (V(x,y') - V_0(x)) dy'
 \end{equation} 
 We notice that by construction, $V_0(x) - V(x,y) $ has $y$-mean equal to $0$ and consequently the function $q$ defined in~\eqref{eq.norm2} is periodic.
 \end{proof}
\begin{cor}\label{cor.1}
There exists operators 
\[ W =  W  ( x , y , h D_x , h D_y ) \in \Psi^0 (
\TT^2 ) \,, \ \  R = R ( x, y , h D_x , hD_y ) \in \Psi^{0 } ( \TT^2 )
\,,  \]
 such that
\begin{equation}
\label{eq:22}
\begin{split} 
& (\Id +h Q) \left(  D_x ^2 + D^2_y + V ( x, y ) \right)\chi( hD_x, hD_y)\\
& \ \ \ = \Bigl(\bigl(D_x^2 + D_y^2 + V_0 ( x)\bigr) ( \Id+ hQ) +
W \Bigr) \chi(h D_x,  hD_y)  + h R \,,\\
& \ \ \ \ \ \ \ \ \ \ \ \ \ \ \ \ \ \ \ \ \ \ W ( x, y , 0 , \eta ) \equiv 0 \,. 
\end{split}
\end{equation}
\end{cor}
\begin{proof}
Indeed, the same calculation as above shows that  by symbolic calculus, we can take 
$$W (x,y, \xi, \eta) = \frac 2 \xi i \partial_x q (x, y, \eta) \widetilde
\chi ( \xi, \eta) \,, $$ 
where $ \widetilde \chi \in \CIc ( \RR^2 ) $ is equal to one on the
support of $ \chi $. 
\end{proof}

\medskip
In the case of irrational tori $ \TT^2 \simeq [0,A] \times [0,A] $, $
A/B \notin \Q $, we need slightly more complicated versions of 
Proposition \ref{p:2} and Corollary \ref{cor.1}. They involve
covering $ \TT^2 $ by a strip. 

Let us consider a  constant rational vector field on the torus given 
by a direction  
\[  \Xi_0 = c  ( n A , m B ) \,,  \ \ n , m \in \ZZ \,, \ \ c 
\in \RR \setminus \{ 0 \} \,. \] 
As shown in Fig.~\ref{f:3} we can find a %maximal 
strip bounded in the direction of $ \Xi_0 $ and covering $ \TT^2 $. If 
the torus is itself rational (that is $ A/B \in {\mathbb Q} $ in
\eqref{eq:AB})
we can find a rectangle $ R $ 
with sides parallel to  $ \Xi_0 $ and $ \Xi_0^\perp $
which covers $ \TT^2 $.

\begin{figure}[ht]
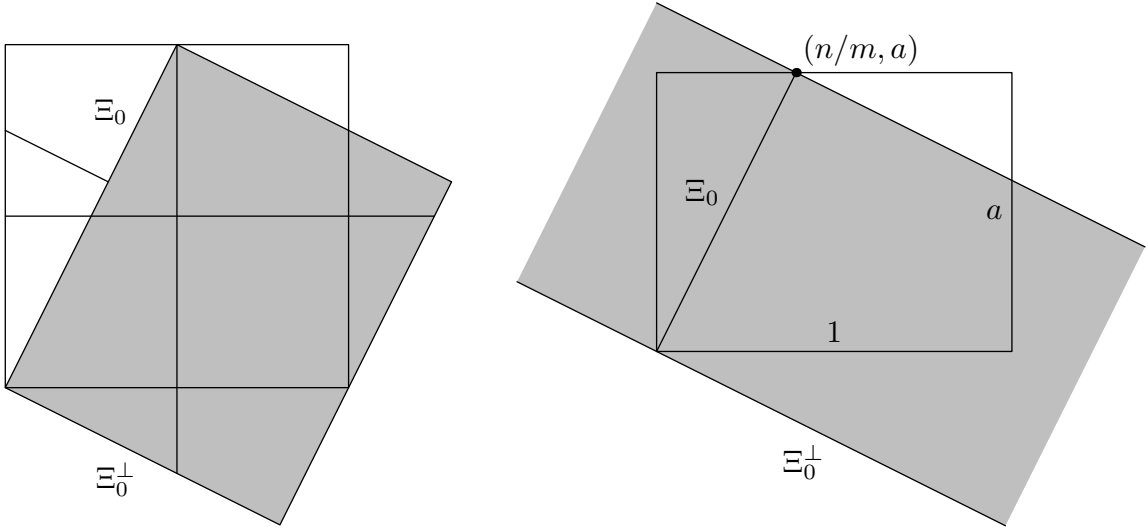

\includegraphics[width=6cm]{m11.mps}\qquad\includegraphics[height=7cm]{m12.mps}
%$$\includegraphics[width=5.5cm]{rectanglebis}\qquad\includegraphics[height=5.5cm]{rectangle}$$
\caption{On the left, a rectangle, $ R $, covering a rational torus $ \TT^2
  $. In that case we obtain a periodic solution on $ R $. On the
  right, the irrational case: the strip with sides $ m \Xi_0
  \times   \RR \Xi_0^\perp $, $ \Xi_0 = ( n/m, a ) $ (not normalized
  to have norm one), 
  also covers the torus $ [ 0 , 1 ] \times [ 0 , a ] $. Periodic
  functions are pulled back to functions satisfying \eqref{eq:per}.}
\label{f:3}
\end{figure} 

Let us normalize $ \Xi_0 $ to have norm one,
\begin{equation}
\label{eq:Xi0}   \Xi_0 = \frac1 { \sqrt { n^2 A^2 + m^2 B ^2 } } ( n A
, m B  )\,,
\ \ \Xi_0^\perp  = \frac1 { \sqrt { n^2 A^2 + m^2 B^2 } } ( - m B ,
n    A  ) \,.\end{equation}
The change of coordinates in  $\RR^2 $, 
\begin{equation}
\label{eq:F}  F \; : \; ( x, y ) \longmapsto  z = F ( x, y )  = x
\Xi_0^\perp  + y \Xi_0 \,, 
\end{equation}
is orthogonal and hence $ - \Delta_z = D_x^2 + D_y^2 $. 

We have the following simple lemma: 
\begin{lem}
\label{l:si}
Suppose that $ \Xi_0 $ and $ F $ are given by \eqref{eq:Xi0} and
\eqref{eq:F}. If 
 $ u = u ( z ) $ is perodic with respect to $ A \ZZ \times B  \ZZ$
then 
\begin{equation}
\label{eq:per}
 F^* u ( x + k a , y + \ell b ) = F^* u ( x , y - k \gamma ) \,, \ \ k,
 \ell \in \ZZ \,,  \ \ ( x, y ) \in \RR^2 \,, 
\end{equation}
where, for any fixed $ p , q \in \ZZ $, % (not both equal to $ 0$), 
\[ a =   \frac {( q n - p m ) A B } { \sqrt {  n^2 A^2 + m^2 B ^2 } }    \,,
\ \      b =  \sqrt {  n^2 A^2  + m^2 B^2 }  \,, \ \  \gamma = -
\frac{ pn A^2 + q m B^2  }{  \sqrt { n^2 A^2 + m^2 B^2 } }  \,. \]
When $ B/A = r/s \in \Q $ then 
\[  F^* ( x + k \widetilde a , y + \ell b ) = F^* u ( x, y ) \,, \ \  \ \ k,
 \ell \in \ZZ \,,  \ \ ( x, y ) \in \RR^2 \,, \]
for $ \widetilde a = ( n^2 s^2 + m^2 r^2 )  a  $.
\end{lem}
\begin{proof}
The proof is a calculation: we need to find $ a, b$, and $ \gamma $ so
that for any $ k , \ell \in \ZZ$ there exist $ P, Q \in \ZZ $ so that 
\[  k  a \Xi_0^\perp + ( \ell b + k \gamma ) \Xi_0 =  P A (1, 0 ) + Q
B ( 0 ,  1)  \,. \]
Taking $  b =  \sqrt { n^2 A^2 + m^2 B^2 } $
we only need to check  
that this relation holds with $ k = 1 $ and $ \ell = 0 $:
\[   a \Xi_0^\perp +  \gamma  \Xi_0 =  p A (1, 0 ) + q
B ( 0 ,  1)  \,, \]
which can be solved for $ a $ and $ \gamma $ for any $ p $ and $ q$.
By taking inner products with $ \Xi_0 $, $ \Xi_0^\perp $ we obtain 
formul{\ae} for $ a $ and $ \gamma $. 

When $ B/A = r/s$, $ r, s \in \ZZ \setminus \{ 0 \} $ we need to find 
$ M \in \ZZ \setminus \{ 0 \} $ so that $ M \gamma = K b $ for some $
K \in \ZZ $. We check that $ M = n^2 s^2 + r^2 m^2 $ works and hence
we obtain periodicity.
\end{proof} 

\medskip
\begin{rem}
Condition \eqref{eq:per} for $ w = F^*u $ is in fact equivalent to 
periodicity of $ u $ with respect to 
\[ \ZZ\,  \vec v_1  \, \oplus \, \ZZ \, \vec v _2  \,, \ \ \  \vec v_1 = ( n A , m B ) \,,
\ \vec v_2 := ( p A , q B ) \,. \]
 That periodicity is of course implied by periodicity 
with respect to $ A \ZZ \times B \ZZ$.
\end{rem}
\medskip

\begin{rem}
A natural choice of $ p $ and $ q $ which excludes the degenerate
cases $ p = q = 0 $ and $ p = n $, $ q = m $, can be obtained by 
assuming (without loss of generality) that $ n $ and $ m $ are
relatively prime and then taking $ p $ and $ q $ satisfying 
\[  n q - m p = 1 \,, \]
which is possible by Bezout's theorem. This will be the choice
we make in what follows.
\end{rem}

\medskip

We can now give a generalized version of Proposition \ref{p:2}:
\begin{prop}
\label{p:2b} 
Suppose that $ F : \RR^2 \rightarrow \RR^2 $ is given by
\eqref{eq:F}  and that $ V \in \CI ( \RR^2 ) $ is periodic
with respect to $ A \ZZ \times B \ZZ $. Let $ a , b $ and $ 
\gamma $ be as in \eqref{eq:per}. 

Let  $\chi \in \CIc ( \mathbb{R}^2 ) $ 
is equal to $0$ near the set $\eta =0$.
There exist operators
\[   Q( x, y , hD_y ) \in \CI ( \RR ) \otimes \Psi^0 ( \RR )  \,, \ \
R ( x, y ,hD_y , h D_x )
\in  \Psi^{0 } ( \RR^2  )  \,, \]
such that 
$  (F^{-1})^*  Q F^* $ and $  (F^{-1})^*  R F^* $ preserve 
$ A \ZZ \times B \ZZ $ periodicity, and
\begin{equation}
\label{eq.normale}
\begin{split}
& ({\Id} + h Q) \left(D^2_y + F^* V ( x, y ) \right)\chi( h D_x, hD_y)\\
& \ \ \ \ \ \ \ \ = (D_y^2 + V_0 ( x)) ({\Id} + hQ) \chi( h D_x,  hD_y) + hR % ( x, y , hD_y) 
  \,, \end{split}
\end{equation}
where 
\[  V_0 ( x) := \frac 1 { b } \int_0^b F^* V ( x , y ) dy 
\,,\]
satisfies $ V_0 ( x + k a ) = V_0 ( x ) $, $ k \in \ZZ $. 
\end{prop}
\begin{proof}
We proceed as in the proof of Proposition \ref{p:2}. We need to solve
equation \eqref{eq.norm} but now $ q $ has to satisfy the twisted 
periodicity condition \eqref{eq:per}, which will follow if 
$ q_F ( z , \eta ) :=  (F^{-1}) ^*q ( \bullet, \eta ) ( z ) $ is $ A\ZZ \times B \ZZ $
periodic. The equation \eqref{eq:per} is then equivalent to 
\begin{equation}
\label{eq:Xi1} 
n  \eta \langle \Xi_0, \partial_x \rangle q_F ( z , \eta ) = V ( z ) - 
 (F^{-1})^* V_0 ( z ) \,, \end{equation}
on the support of $ \chi ( \xi, \eta ) $.  We note that $ (F^{-1})^*
V_0 $ is the average of $ V $ over the (closed) orbit of $ \langle
\Xi_0, \partial_z \rangle $.  In particular, the average of the 
right hand side is $ 0 $. 

An equation of this form can be 
solved on any compact Riemannian manifold: if $ X $ is a length
one vectorfield with closed integral curves, and $ f $ is 
function integrating (with respect to the length parameter) to $ 0 $
along those curves,
then there exists $ u $, smooth on $ M $, satisfying $ X u = f$. 
To see this we solve the equation on each curve, demanding that
$ u $ integrates to zero on that curve. That determines $ u $ 
uniquely and hence provides a global smooth solution. Note that
this is not the solution we took in \eqref{eq.norm2}. 
In the notation of \eqref{eq.norm2} the current 
solution corresponds to 
\begin{gather*}
   q(x, y, \eta) = \frac{ i \zeta (\eta)} { 2 \eta} \left( \int_0^y (V(x,y)
- V_0(x)) dy  - q_0 ( x ) \right) \,, \\ q_0 ( x ) := \frac 1 B
\int_0^B \int_0^y  (V(x,y')
- V_0(x)) dy' dy \,. 
\end{gather*}
\end{proof}

Finally, we have the corresponding analogue of Corollarry \ref{cor.1}.
\begin{cor}\label{cor.2}
In the notation of Proposition \ref{p:2b}, there
exists operators 
\[ W =  W  ( x , y , h D_x , h D_y ) \in \Psi^0 (
\RR^2 ) \,, \ \  R = R ( x, y , h D_x , hD_y ) \in \Psi^{0 } ( \RR^2 )
\,,  \]
 such that
$  (F^{-1})^*  W F^* $ and $  (F^{-1})^*  R F^* $ preserve 
$ A \ZZ \times B \ZZ $ periodicity, and 
\begin{equation}
\label{eq:22b}
\begin{split} 
& (\Id +h Q) \left(  D_x ^2 + D^2_y + V ( x, y ) \right)\chi( hD_x, hD_y)\\
& \ \ \ = \Bigl(\bigl(D_x^2 + D_y^2 + V_0 ( x)\bigr) ( \Id+ hQ) +
W \Bigr) \chi(h D_x,  hD_y)  + h R \,,\\
& \ \ \ \ \ \ \ \ \ \ \ \ \ \ \ \ \ \ \ \ \ \ W ( x, y , 0 , \eta ) \equiv 0 \,. 
\end{split}
\end{equation}
\end{cor}

\section{A semiclassical estimate}
\label{se}

The purpose of this section is to prove the main step towards
Theorem~\ref{t:2}, its semiclassically localized version:
\begin{prop}\label{th:3}
Let $\chi\in \CIc (  -1 , 1 )$ be equal to $1$ near $0$, and define
$$ \Pi_{h, \rho} (u_0)  := \chi \left( \frac{ h^2(- \Delta + V) -1}
  {\rho} \right) u_0\,,  \ \ \ \rho > 0 \,. $$
Then for any $T>0$ there exists $\rho, C, h_0 >0$ such that for any
$0<h<h_0$, $u_0$, 
%and any solution to~\eqref{eq:dy} with initial data $\Pi_{h,
%\rho}(u_0)$ 
we have  
\begin{equation} \label{eq:semi}
 \|\Pi_{h, \rho} u_0\|_{L^2}^2 \leq C \int_{0}^T \| e^{-i t ( - \Delta +
   V )} \Pi_{h, \rho} u_0\|_{L^2 ( \Omega ) }^2 dt \,.
\end{equation} 
\end{prop}
\begin{proof}
We first observe that if the estimate~\eqref{eq:semi} is true for some $\rho>0$,
then is is true for all $0<\rho'<\rho$. As a consequence, if
~\eqref{eq:semi} were false, there would exist $ T > 0 $ and  sequences 
$$ h_n \longrightarrow 0, \quad \rho_n \longrightarrow 0, \quad u_{0,n}=
\Pi_{h_n, \rho_n} (u_{0,n}) \in L^2,  $$
$$  i \partial_t u_n ( t, z ) = ( - \Delta + V( z ) ) u_n ( t, z )
\,, \ \  u_{n} ( 0 , z ) = u_{0,n} ( z ) \,, $$
such that 
$$ 1= \| u_{0,n} \|^2_{L^2} , \qquad \int_{0}^T
\|u_n ( t , \bullet ) \|_{L^2(\Omega)}^2 dt \longrightarrow 0 \,. $$
The sequence $(u_n)$ is bounded in $L^2_{\rm{loc}} ( \mathbb{R} \times
\T^2)$ and consequently, after possibly extracting a subsequence,
there exists  a semi-classical defect measure $\mu$ on $\R_t \times T^* (\T^2_z)$ such
that for any function $\varphi \in 
{\mathcal C}^0_0 ( \R_t)$ and any $a\in \CIc ( T^*\T^2_z)$, we have 
\[\begin{split} \langle \mu, \varphi(t) a(z, \zeta)\rangle &   = \lim_{n\rightarrow  \infty} \int_{\R_t \times \T^2} \varphi(t) 
( a(z, h_n D_z)  u_n ) ( t , z ) \overline u_n (t,z) dt dz \,.
\end{split} \]
%Here we also defined the measure $ \mu_\varphi $ on $ T^* \TT^2 $, for
%any fixed $ \varphi \in {\mathcal C}_0^0 ( \RR ) $.

\renewcommand\thefootnote{\ddag}%

Furthermore, standard arguments\footnote{see \cite{AM} for a review of recent 
results about measures used for the Schr\"odinger equation.} 
show that the measure $\mu$ satisfies
\begin{itemize}
\item 
\begin{equation}\label{eq.nonvan}
 \mu( (t_0 , t_1 ) \times T^*\T^2_z) = t_1 - t_0  \,. 
\end{equation}
\item  The measure $\mu $ on $\R_t \times T^* (\T^2)$ is supported in the set 
$$\{(t,z, \zeta) \in \R_t \times \T^2_z \times \R^2_\zeta; |\zeta|=1\}$$
and is  invariant  under the action of the geodesic flow:
$$ 2 \langle \zeta , \partial_z \rangle \,  \mu_\varphi  =0.$$
We shall only use that the support of the measure $\mu$ is
invariant: %for almost every $ t_0 $
\begin{equation}
\label{eq:invm} (t_0, z_0, \zeta_0) \in \text{supp} ( \mu) \ \Longrightarrow \ (t_0,
z+ s\zeta_0, \zeta_0) \in \text{supp} ( \mu)\,, \ \  \forall s\in \R
\end{equation}
\item The measure $\mu$ vanishes on $(0,T) \times T^*\Omega$. 
\end{itemize}
We are going to show that the measure $\mu$ is identically equal to
$0$ on $(0,T) \times T^* \TT^2 $.
This will provide a contradiction  with~\eqref{eq.nonvan}.

\medskip

\begin{rem}\label{rem.leb}
In the case of geometric control, as in the work by Lebeau, the
vanishing of 
$\mu\mid_{(0,T)}$ is a direct consequence of the invariance property. 
Actually, in Lebeau's work, which concerns boundary value problems,
the difficult part is to precisely 
to prove (analogues of) this invariance property. See the appendix 
for more details.
\end{rem}

\medskip

%Let us fix $\varphi \in \CIc ( 0,T)$ and define 
%$$\mu_\varphi = \int_{t} \varphi d\mu.$$ 
The $z$ projection of a trajectory associated to a irrational
direction $\zeta$ is dense. Consequently, the support of
$\mu\mid_{t\in (0, T)} $ contains only points 
$(t, z , \Xi_0)$ with rational $\Xi_0$:
\begin{equation}
\label{eq:AB} \TT^2 \simeq [0, A ]_x \times [0 , B ]_y \,, \ \  \Xi_0 = \alpha ( A/B ,
n/m ) \,, \ \ n, m \in \ZZ \,, \ \ \alpha \in \RR \setminus \{ 0 \}
\,.  \end{equation}
In fact, that is the condition implying that the trajectory $  
s \longmapsto z_0 + s \Xi_0 $ is closed when projected to $ \TT^2 $,
for any $ z_0 \in \RR^2 $. Any other trajectory is dense. 

Define 
\begin{equation}
\label{eq:Mz}
M_\mu :=  \pi_1 ( \supp \mu \cap \{(t,z, \zeta); t\in (0,T)\}) \,, \ \ \ \pi_1 : ( t , z , \zeta )
\longmapsto \zeta \,. 
\end{equation}

The discussion above shows that $ M $ contains only rational
directions  and hence it is countable and closed.
This in turn implies that it contains an isolated point, $\Xi_0$ (perfect sets cannot be countable). 

We now consider the Schr\"odinger equation on 
on the strip (or rectangle)  $ R = \RR_x \times [0,
b]_y $  ($ R = [0, a]_x \times [0,
b]_y $, respectively) using the function $ F $ given in \eqref{eq:F}.
In this coordinate system, $\Xi_0 = (0,1)$ -- see Fig.~\ref{f:3}.

Let $ \chi(hD_z )$ be a Fourier multiplier with a symbol 
supported in a neighborhood of $\Xi_0$ containing no other points in
the intersection with $(0,T) \times T^* \T^2 $ of the support of $\mu$, and define
$$ \widetilde u_n = \chi (h_n D_z) u_n \,,  $$ 
We denote by $\widetilde \mu$, the semiclassical
measure of the sequence $\widetilde{u}_n$.
%,  and \[  \widetilde\mu_\varphi= \int_t \varphi d\mu \,. \]
 We clearly have
\[ \widetilde \mu = (\chi (\zeta))^2\mu \,, \]
%$\widetilde{\mu} _\chi = (\chi(\zeta))^2 \mu_\chi$,
 and consequently, we know that the $\zeta$-projection, $ \pi_1 $, of
 the  intersection with $(0,T) \times T^* \T^2 $ of the support
 of the measure $\widetilde{\mu}$ is equal to $\{\Xi_0\}$:
\begin{equation}
\label{eq:mut} 
 M_{\widetilde \mu} = \{ \Xi_0 \} = \{ ( 0 , 1 ) \} \,,
\end{equation}
where we used the coordinates  $( x, y ) $ in the last identification.
%For the remainder of this section we will use that coordinate system.

Using Proposition \ref{p:2b} (or, in the easier case of rational tori,
Proposition \ref{p:2}) we define
$$ v_n = \Bigl(1+hQ\Bigr) \widetilde u_n \,. $$
Since the operator $Q$ is bounded on $L^2$,  
the semiclassical defect measures associated to $v_n$ and
$\widetilde{u}_n$ are equal
% (and coincide with $\chi^2 (\xi,
%\eta)\mu$). 
We now consider the time dependent Schr\"odinger equation satisfied 
by $ v_n $. With 
\[ Q_n := Q ( x, y , h_n D_y ) \,, \ \ R_n := R ( x, y , h_n D_x, h_nD_y ) \,, \
\ W_n := W ( x, y , h_n D_x , h_nD_y ) \,, \]
given in \eqref{eq:22b} and $ \chi_n := \chi ( h_n D_z ) $, we have 
\begin{equation}\label{eq:v}
\begin{split}
 (i \partial_t + \Delta -V_0(x) ) v_n & = ({\Id} + h_n Q_n ) (i \partial_t +
 \Delta - V(x,y) ) \chi  u_n  -W_n   \chi   {u}_n - h_n R_n u_n  
 \\
&  = -W_n  \chi_n  {u}_n + [V, \chi ] u_n + o  _{L^2} (1) \\
&  = -W_n \chi_n  {u}_n+o _{L^2_{x,y}} (1) 
\end{split} 
\end{equation}
We also recall that according to Corollary~\ref{cor.1},  the symbol of the operator $W$ vanishes in the set 
 $$\{x,y,\xi, \eta) \; : \;  \xi=0\} \,. $$ 
Consequently it vanishes on the  intersection with $(0,T) \times T^* \T^2
$ of 
the support of the defect measure of 
$\chi_n  {u}_n= \widetilde{u}_n$ which,  by construction, is  included in the set 
 $$   \pi_1^{-1} ( M_{\widetilde \mu} ) =  \{(t, x,y,\xi, \eta) \; : \;  \xi=0\,, \ \eta =1\}.$$
 As a consequence, the semiclassical measure 
of $W_n  \chi_n   {u}_n$ is equal to $0$. This implies that
 \begin{equation}\label{eq:vbis}
 (i \partial_t + \Delta -V_0(x) ) v_n = o  _{L^2_{\rm{loc}}  (
     (0,T)\times R ) } (1)\,. 
     \end{equation}

In view of Lemma \ref{l:si}, see \eqref{eq:per}, we are nowin the setting of Proposition \ref{p:1bis}. 
To apply it let us choose a band domain ${\omega}
 = \omega_x \times [0, b ]_y $
 where 
$ \omega_x $ is a an interval 
 such that any line $\{x\} \times [0, b]_y $,  encounters
 the interior of $ \pi_2^{-1} ( \Omega) $, where $ 
\pi_2 : R \rightarrow \TT^2 $ -- see Fig.~\ref{figure-2}.

We know that there exists 
$ ( t_0 ,  z_0, \Xi_0 )\in \text{supp} ( \mu)$  for some $t_0 \in
(0,T) $. %We define $ ( x_0 , y_0 )$  by  $ z_0 = F ( x_0,
%y_0 ) $ (it is not uique). 

\begin{figure}[ht]\label{figure-2}
\includegraphics[height=7.5cm]{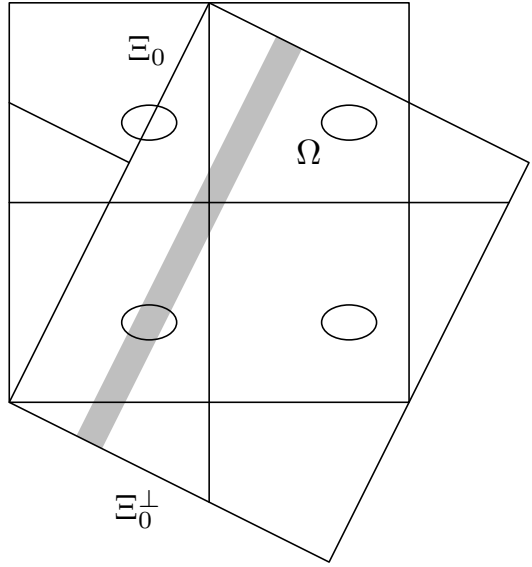}
%$$\ \hskip 3cm  \includegraphics[width=10cm]{rectangleter.pdf}$$
%$$\includegraphics[width=5.5cm]{rectanglebis}\qquad\includegraphics[height=5.5cm]{rectangle}$$
\caption{The rectangle $ R $, covering a rational torus $ \TT^2  $ and
the choice of $ \omega = \omega_x \times [0, b]_y $ shown as a shaded region.}
\end{figure}

 %We then take $ \psi \in \CIc ( (
%\alpha_1, \beta_2 ) )$, equal to one in $ [ \beta_1, \alpha_2 ] $. 

Since according to~\eqref{eq:vbis} on $(0,T)$, the family $(v_n)$ is a
family of solutions of the free Schr\"odinger equations
up to  $o _{L^2_{\rm{loc}} ( ( 0 , T ) \times R ) } (1) $,
we can apply Proposition~\ref{p:1bis} to obtain
\[ \begin{split}   \| v_n\|^2_{L^2 ( (t_0 - \epsilon, t_0 + \epsilon) \times \TT^2)}
 & \leq  \| v_n\|^2_{L^\infty ( (t_0 - \epsilon, t_0 + \epsilon)   \times \TT^2) }  \\
& \leq C \int_{t_0 - 2 \epsilon} ^{t_0 + 2 \epsilon} \int_{{\omega}} |v_n (t,z) |^2 dz +o(1)
\end{split} \]
where $\epsilon >0$ is chosen small enough so that 
 $( t_0 - 2\epsilon , t_0 + 2\epsilon ) \Subset (0, T)$.
This implies that there exists $t'_0\in ( t_0 - \epsilon, 
t_0 + \epsilon) $, $z'_0\in {\omega}$, $\Xi'_0$ such that  
$$ (t'_0, X'_0, \Xi'_0) \in \text{supp}( \widetilde{\mu})\,. $$
From \eqref{eq:mut}  we necessarily have $\Xi'_0= \Xi_0$. The invariance of the
support of $ \widetilde \mu $ shows that the whole line 
$$(t'_0, z'_0+ s \Xi_0, \Xi_0)\in \text{supp}( \widetilde{\mu}) \,. $$
consequently  the support of the measure $\widetilde{\mu}$ does encounter  the set $(0,T) \times T^*\Omega$, which gives the contradiction and concludes the proof of Proposition~\ref{th:3}. \end{proof}

\section{Proof of Theorem \ref{t:2}}
\label{fse}

To prove Theorem \ref{t:2} we need to pass from the semiclassical
estimate of \S \ref{se} to an estimate for all frequencies. 
%\subsection{Gluying semi-classical estimates}  
We start with a result involving an error term:
\begin{prop}\label{prop.4.1}
For any $T>0$ and any non empty open set $\Omega \subset \T^2$, there exists $C>0$ such that for any $u_0 \in L^2(\T^2)$,
\begin{equation}
\label{eq:p41} \|u_0 \|^2_{L^2} \leq C \Bigl( \int_0^T \int_\Omega
|e^{-it (- \Delta + V)} u_0 |^2 dz dt + \|u_0 \|_{H^{-2}}^2\Bigr)
\,. 
\end{equation}
\end{prop}
\begin{proof} 
Consider a partition of unity 
\begin{gather*}  1= \varphi_0 (r)^2 + \sum_{j=1}^{\infty}
\varphi_j ( r )^2 \,, \ \ \  \varphi_j ( r ) :=   \varphi(R ^{-j} |r| )
\,, \ \  R > 1 \,, \\
\varphi \in \CIc ( ( R^{-1} , R ) ; [0, 1]  ) \,, \ \  ( R^{-1}, R) \subset \{ r \;
: \;  \chi( r / \rho)  \geq 1/2 \} \,, \end{gather*}
where $ \chi $  and $ \rho $ come from Proposition~\ref{th:3}.
 Then,  we decompose $ u_0 $ dyadically:
\[ \|u_0 \|^2_{L^2} = \sum_{j=0}^\infty  \|\varphi_j( P_V ) u_0
\|_{L^2}^2 \,. \ \ \  P_V := - \Delta + V \,. \]

Let 
$ \psi \in \CIc ( ( 0 , T ) \; [ 0, 1 ] ) $,  $\psi ( t ) > 1/2$, 
or $ T/3 < t < 2T / 3 $.
We first observe that in 
Proposition \ref{th:3} we have actually proved (see the remark after
Proposition \ref{p:1bis}) that 
\begin{equation}
\label{eq:piece}   \|\Pi_{h} u_0\|_{L^2}^2 \leq C \int_\RR \psi( t )^2 \| e^{-i t ( - \Delta +
   V )} \Pi_{h} u_0\|_{L^2 ( \Omega ) }^2 dt \,, \ \ 0 < h < h_0
 \,,\end{equation}
which is the version we will use.

Taking $K$ large enough so that $R^{-Kj} \leq h_0$, where $ h_0 $ is
given above
we apply \eqref{eq:piece} to the dyadic pieces:
\[ \begin{split}
  \|u_0 \|^2_{L^2} & =\sum_j \|\varphi_j( P_V)u_0\|_{L^2}^2 \\
&   \leq \sum_{j=0}^K  \|\varphi_j( P_V )u_0\|_{L^2}^2 +C
\sum_{j=K+1}^\infty  \int_{0}^T \psi ( t ) ^2 \|  \varphi_j(  P_V)
e^{-it P_V } u_0\|_{L^2( \Omega) } ^2 dt \\
& = \sum_{j=0}^K  \|\varphi_j( P_V )u_0\|_{L^2}^2 + C
\sum_{j=K+1}^\infty  \int_\RR \| \psi ( t ) \varphi_j(  P_V)
e^{-it P_V } u_0\|_{L^2( \Omega) } ^2 dt \,. 
\end{split} \]
Using the equation we can replace $ \varphi ( P_V ) $
by $ \varphi ( D_t ) $ which meant that we did not change the domain 
of $ z $ integration. We need to consider the commutator of 
$  \psi \in \CIc ( ( 0 ,  T) ) $ and $ \varphi_j ( D_t ) = \varphi ( R^{-j}  D_t ) $. 
If  %$ \widetilde \varphi \in \CIc $ 
$ \widetilde \psi $  is equal to $ 1 $ on 
$ \supp  \psi$ then the semiclassical pseudodifferential 
calculus with $ h = R^{-j}$ (see for instance \cite[Chapter 4]{EZB}) gives
\begin{equation}
\label{eq:wideps}     \psi ( t )  \varphi_j ( D_t )  =    \psi ( t )  \varphi_j
(D_t )   \widetilde  \psi ( t ) +  E_j ( t, D_t) \,, \ \ 
\partial^\alpha E_j  = {\mathcal O}   ( \langle t \rangle^{-N } \langle
\tau \rangle^{-N} 
R^{- N j } ) \,, 
\end{equation}
for all $ N$ and uniformly in $ j $.

The errors obtained from $ E_j $ can be absorbed into the $ \| u_0
\|_{ H^{-2}  ( \TT^2 ) } $ term on the right hand side. Hence we obtain
\[ \begin{split}
  \|u_0 \|^2_{L^2} & 
\leq C\|u_0 \|_{H^{-2} ( \T^2)} ^2+ C \sum_{j=0}^\infty \int_{0}^T 
 \|   \psi (  t)  \varphi_j(  D_t  ) e^{- it P_V }u_0 \|^2_{L^2( \Omega)} dt
\\
& \leq
\widetilde C\|u_0 \|_{H^{-2} ( \T^2)} ^2+ C   \sum_{j=0}^\infty
\langle   \varphi_j(  D_t  )^2  \widetilde \psi ( t ) e^{-it P_V }u_0,  \widetilde \psi( t )
e^{-it P_V}u_0,  \rangle_{L^2 ( \RR_t \times \Omega)} \\
& = \widetilde  C\|u_0 \|_{H^{-2} ( \T^2)} ^2 + C \int_\RR
  \| \widetilde \psi ( t ) e^{-it P_V } u_0\|^2_{L^2( \Omega)} dt  \\
& \leq \widetilde C\|u_0 \|_{H^{-2} ( \T^2)} ^2 + C \int_{0}^T 
  \| e^{-it P_V } u_0\|^2_{L^2( \Omega)} dt  
\end{split}
\]
where the last inequality is the statement of the proposition. 

\end{proof}
%\subsection{Eliminating the error terms}\label{se.uniq}

To eliminate the $ H^{-2} $ error term in \eqref{eq:p41} 
we use the now classical uniqueness-compactness argument of Bardos,
Lebeau and Rauch~\cite{BLR}. For reader's convenience we recall the 
argument.

Let us fix $\delta \geq 0$ and define 
\[ N_\delta := \{ u_0 \in L^2(\T^2) \; : \; 
 e^{-it (- \Delta + V)} u_0 \equiv 0 \text{ on } (0  ,T- \delta )\times \Omega\} \,.\]
 Let $u_0 \in N_0$. We now define
\[  v_{\epsilon, 0} = \frac1 \epsilon \left(  e^{-i \epsilon ( - \Delta
    + V ) } - \Id\right) { u_0 }  \,. \]
If $\epsilon \leq \delta$, then 
$   e^{ -t ( - \Delta + V ) } v_{\epsilon, 0 } \equiv 0 $  on $
  ( 0 , T-\delta  ) \times \Omega $

We write $u_0$ in terms of othonormal
eigenvectors of $-\Delta + V$:
$ u_0 = \sum_{\lambda \in \sigma ( - \Delta+ V)} u_{0, \lambda}
e_{\lambda} \,. $ 
Proposition~\ref{prop.4.1} applied with $T$ replaced by $T/2$ gives that for any $0< \alpha, \beta <T/2$, we have 
\[ \begin{split}
 \|v_{\alpha, 0} - v_{\beta, 0}\|_{L^2}^2 & \leq C \| v_{\alpha , 0} 
 v_{\beta, 0 } \|_{H^{-2}} ^2 \\
& \leq C \sum_{\lambda \in \sigma ( - \Delta+ V)} \Bigl| \frac{ e^{-i
    \alpha \lambda} -1 } { \alpha } - \frac{ e^{-i \beta \lambda} -1 }
{ \beta }\Bigr| ^2 ( 1+ \lambda)^{-2} | u_{0, \lambda}|^2
\\
& 
\leq C' \sum_{\lambda \in \sigma ( - \Delta+ V)}  \lambda^2  | \alpha
- \beta |^2  ( 1+ \lambda)^{-2} | u_{0, \lambda}|^2 \leq C' | \alpha - \beta|^2 \,. 
\end{split}
\]
Hence
$ \lim_{\alpha, \beta \rightarrow 0} \|v_{\alpha, 0} - v_{\beta,
  0}\|_{L^2} =0 $, 
and there exists $ v_0 \in L^2 $ such that 
\[  L^2\text{-}\lim_{\alpha \to 0 }  v_{\alpha,0}  = v_0 \,. \]
This limit is necessarily in $N_\delta $ for all $
\delta  >0$, hence in $N_0$. 
On the other hand, we  have in the sense of distributions,
$$e^{-it(- \Delta + V )} v_0= \partial _t e^{-it(- \Delta + V )} u_0
\,, $$
which implies that 
\[ v_0= -i(- \Delta + V ) u_0 \,. \]
Hence $N_0 $ is an invariant
subspace of $-i(- \Delta + V )$.  According to Proposition~\ref{prop.4.1}, $\|u_0\|_{H^{-2}}$ is a norm
on $N_0 $ which is consequently of finite dimension. This means that there exists
an eigenvector 
$w$,
$$ (- \Delta + V )w= \mu w\,, \ \ w|_{\Omega} =0 \,. $$
We can now use the 
the standard unique continuation results 
for elliptic second order operators to conclude that $w \equiv 0$
which then implies that $N_0= \{0\}$. 

Finally, to conclude the proof of Theorem \ref{t:2}, we argue by contradiction: if~\eqref{eq:t2} were not true, we could construct a sequence $(u_{n,0})\in L^2( \T^2)$ such that 
$$ 1= \|u_{n,0} \|_{L^2} , \qquad \int_{0}^T \int_{\Omega}
\bigl|e^{-it(-\Delta+V)} u_{n,0}\bigr| ^2 dx dt 
\longrightarrow 0 \,, \ \  n\longrightarrow \infty \,. $$
We could then extract a subsequence $u_{n_k,0}$ converging weakly in $L^2$ (and hence strongly in $H^{-2}$) to a limit $u_0\in N$ which would satisfy, according to Proposition~\ref{prop.4.1},
\[ 1= \lim_{k\rightarrow \infty}\| u_{n_k,0} \|_{L^2}  \leq C 
  \int_{0}^T \int_{\Omega} \bigl|e^{-it(-\Delta+V)} u_{n_k,0}\bigr| ^2
  dx dt +C \|u_{n_k,0}\|_{H^{-2}} ^2 \,. 
\]
That would imply that 
\[ 1 \leq  C \lim_{k\rightarrow \infty}  \|u_{n_k,0}\|_{H^{-2}} ^2= C
\|u_0\|_{H^{-2}}^2 \,, \]
showing that there exists $u_0 \in N, u_0 \not \equiv 0$ contradicting
our earlier conclusion.
This ends the proof of Theorem~\ref{t:2}

\vspace{0.5cm}
\begin{center}
\noindent
{\sc  Appendix: {Proof of Lemma \ref{lem.2.3}}}
\end{center}
\vspace{0.4cm}
\renewcommand{\theequation}{A.\arabic{equation}}
\refstepcounter{section}
\renewcommand{\thesection}{A}
\setcounter{equation}{0}

To prove \eqref{eq:circ}, we rewrite it as an inequality for 
periodic functions, that is as an inequality on the circle:
\begin{equation}
\label{eq:perA}
\| v_0 \|_{L^2 ( \TT^1 ) }^2 \leq C \int_0^T \| 
e^{ - i t ( D + \beta )^2 + W ) } v_0 \|_L^2 ( \omega_x ) ^2 dt
\,.  \end{equation}
As presented in detail in the second part of \S \ref{fse}, this
%follows from the analogue of estimate \eqref{eq:p41}
follows from the reduction performed in~\eqref{eq.redu} and 
the analogue of estimate \eqref{eq:p41}: 
there exist $C>0$ such that for any $\beta \in [0, 2\pi/a]$, and any $v_0 \in L^2( 0, a)$,
\begin{equation}
\label{eq:perAA}
\|v_0 \|^2_{L^2 ( \TT^1 ) } \leq C \Bigl( \int_0^T
\int_{\omega_x} 
|e^{-it (( D + \beta) ^2  + W )} v_0 |^2 dz dt + \|v_0 \|_{H^{-2} (
  \TT^1 ) }^2\Bigr) 
\,. 
\end{equation}
We remark that the proof in~\S \ref{fse} applies for this setting
where we consider a family of operators, $(D_x + \beta^2 +V, \beta\in
[0, 2\pi/a]$, it could actually handle the more general case  of a
family of potentials $V$,  relatively compact in $L^\infty$.

As shown in Proposition \ref{prop.4.1} this in turn follows from the analogue  
Proposition \ref{th:3}:
for any $T>0$ there exists $ C, h_0 >0$ such that for any
$ \beta \in [ 0 , 2\pi/a  ]  $, $0<h<h_0$, and $v_0 \in L^2 ( \TT^1 ) $,  we have 
\begin{equation} \label{eq:semiA }
 \|\Pi_{h, \beta} v_0\|_{L^2( \TT^1) }^2 \leq C \int_{0}^T \| e^{-i t (
   ( D + \beta  )^2 +
   W )} \Pi_{h, \beta } v_0\|_{L^2 ( \omega_x  ) }^2 dt \,,
\end{equation}
where now, in the notation of Proposition \ref{th:3},
\[  \Pi_{h, \beta} v_0   := \chi \left( { h^2(( D + \beta )^2
    + W ) -1}
\right) v_0\,,  \ \ \beta \in [ 0 ,
2 \pi / a  ] \,.  \]
If this were false there would exist $ T > 0 $ and sequences 
$$ h_n \longrightarrow 0 \,, \quad 
\beta_n \longrightarrow \beta \in [ 0 , 2 \pi/ a ] \,, 
 \quad 
v_{0,n}=
\Pi_{h_n, \beta _n} (v _{0,n}) \in L^2,  $$
$$  i \partial_t v_n ( t, x ) = ( ( D + \beta_n )^2 + W ( x ) ) v_n ( t, x )
\,, \ \  v_{n} ( 0, x ) = v_{0,n} ( x ) \,, $$
such that 
\begin{equation}
\label{eq:vn}  1= \| v_{0,n} \|^2_{L^2 ( \TT^1 ) } , \qquad \int_{0}^T
\|u_n ( t , \bullet ) \|_{L^2(\omega_x)}^2 dt \longrightarrow 0 \,. 
\end{equation}
We associate to the sequence $ v_n $ a semiclassical defect measure, $
\nu $, on $ \RR \times T^* \TT^1 $.  As recalled in \S \ref{se}
(see \cite{AM} and \cite{Mac}) the measure satisfies 
$   \nu ( ( t_0 , t_1 ) \times T^* \TT^1 ) = t_1 - t_0 $, and its
support is invariant under the flow of principal symbol of
$ ( D + \beta )^2 + W ( x ) $ (since $ \beta_n = \beta + o ( 1 )$):
\[ (t_0, x_0, \xi_0) \in \text{supp} ( \nu) \ \Longrightarrow \ (t_0,
x_0 + s\xi_0, \xi_0) \in \text{supp} ( \nu)\,, \ \  \forall s\in \RR
\,. \]
In view of the second part of 
\eqref{eq:vn} the measure $ \nu $ vanishes
on $ ( 0 , T ) \times T^* \omega_x $ which contradicts the invariance
of the support.

\end{document}